\documentclass[a4paper]{amsart}
\usepackage{amssymb,amsmath,enumerate}
\usepackage{txfonts}
\newcommand{\C}{\mathbb{C}}

\newcommand{\co}{\colon\thinspace}

\theoremstyle{plain}
\newtheorem{theorem}{Theorem}[section]
\newtheorem{corollary}[theorem]{Corollary}
\newtheorem{proposition}[theorem]{Proposition}
\newtheorem{lemma}[theorem]{Lemma}
\theoremstyle{definition}

\title{CR regular embeddings of $S^{4n-1}$ in $\C^{2n+1}$}
\author{Naohiko Kasuya}
\address{Naohiko Kasuya: 
Department of Mathematics, Kyoto Sangyo University,
Kamigamo-Motoyama, Kita-ku, Kyoto, 603-8555, Japan.}
\email{nkasuya@cc.kyoto-su.ac.jp}%
\date{}
\begin{document}\sloppy
\maketitle
\begin{abstract}
Ahern and Rudin have given an explicit construction of a totally real embedding of $S^3$ in $\C^3$. 
As a generalization of their example, 
we give an explicit example of a CR regular embedding of $S^{4n-1}$ in $\C^{2n+1}$. 
Consequently, we show that the odd dimensional sphere $S^{2m-1}$ with $m>1$  
admits a CR regular embedding in $\C^{m+1}$ if and only if $m$ is even.  
\end{abstract}

\section{Introduction}\label{sect:introduction}
Suppose $F\co M^n \to \C^q$ is a smooth embedding of an $n$-manifold in $\C^q$. 
Then, for any point $x\in M^n$ and the standard complex structure $J$ on $\C^q$, the following inequality holds: 
\[
\mathrm{dim}_\mathbb{C}(dF_x(T_xM^n)\cap JdF_x(T_xM^n))\geqq n-q. 
\]
If the equality holds for each point $x\in M^n$, the embedding $F$ is called a CR regular embedding, 
and when $n=q$, we say that $F$ is a totally real embedding and $F(M^n)$ is a totally real submanifold. 

Totally real submanifolds have been investigated by many geometers and topologists.  
Especially, the problem of determining which manifolds admit a totally real embedding in $\C^q$ 
has been widely studied from the viewpoint of the $h$-principle 
(Gromov \cite{MR0420697, MR0413206, MR864505}, 
Lees \cite{MR0410764}, Forstneri\v{c} \cite{MR880125}, Audin \cite{MR966952}). 
On the other hand, Ahern and Rudin \cite{MR787894} have constructed 
an explicit example of a totally real embedding of $S^3$ in $\C^3$. 
In the following, let $z=(z_1, z_2, \ldots, z_m)$ be the coordinates on $\C^m$ 
and we regard $S^{2m-1}$ as the unit sphere in $\C^m$. \\

\begin{theorem}[Ahern-Rudin \cite{MR787894}]~\label{AR}
Let $P(z_1,z_2)=z_2\bar{z}_1\bar{z}_2^2+iz_1\bar{z}_1^2\bar{z}_2$. 
Then, the embedding $F\co S^3\to \C^3$ defined by $F(z_1,z_2)=(z_1,z_2,P(z_1,z_2))$ is a totally real embedding.   
\end{theorem}

CR regular embeddings also have been studied by many authors from various viewpoints  
(Cartan \cite{MR1553196}, Tanaka \cite{MR0145555, MR0221418}, Wells \cite{MR0241690, MR0237823}, 
Lai \cite{MR0314066}, Jacobowitz-Landweber \cite{MR2304588}, Slapar \cite{MR3096869, MR3345510}, 
Torres \cite{MR3503716, MR3777134}, Elgindi \cite{1607.02384}). 
In \cite[Section 5]{MR3871721} and \cite{Erratum}, the author and Takase have worked on 
the problem of determining when the $n$-sphere $S^n$ admits a CR regular embedding in $\C^q$ 
and have given some necessary conditions on $(n, q)$. 
In particular, we have proved that the $(4n+1)$-dimensional sphere $S^{4n+1}$ 
does not admit a CR regular embedding in $\C^{2n+2}$ (\cite[Theorem~5.2~(c)]{Erratum}).  
In this paper, we settle the remaining codimension three case by generalizing Theorem~\ref{AR}. 
The following is our main theorem. \\

\begin{theorem}~\label{main}
Let $$Q(z_1, z_2, \ldots , z_{2n-1}, z_{2n})=\sum_{k=1}^n P(z_{2k-1}, z_{2k}), $$
where $P(x, y)=y\bar{x}\bar{y}^2+ix\bar{x}^2\bar{y}$. 
Then, the embedding $F\co S^{4n-1}\to \C^{2n+1}$ defined by 
$$F(z_1, z_2, \ldots , z_{2n-1}, z_{2n})=\Big(z_1, z_2, \ldots , z_{2n-1}, z_{2n}, Q(z_1, z_2, \ldots , z_{2n-1}, z_{2n}) \Big)$$ 
is a CR regular embedding. 
\end{theorem}

\begin{corollary}
Let $m$ be an integer greater than $1$. 
The odd dimensional sphere $S^{2m-1}$ admits a CR regular embedding in $\C^{m+1}$ if and only if $m$ is even. 
\end{corollary}

\section{Proof of Main Theorem}
For a smooth complex-valued function $f$ on $\C^m$, we use the following notations: 
\begin{eqnarray*}
\partial f=\sum_{j=1}^m \frac{\partial f}{\partial z_j}dz_j, \;
\bar{\partial }f=\sum_{j=1}^m \frac{\partial f}{\partial \bar{z}_j}d\bar{z}_j, \\
\frac{\partial f}{\partial z}=(\frac{\partial f}{\partial z_1}, \ldots , \frac{\partial f}{\partial z_m}), \; 
\frac{\partial f}{\partial \bar{z}}=(\frac{\partial f}{\partial \bar{z}_1}, \ldots , \frac{\partial f}{\partial \bar{z}_m}). \\
\end{eqnarray*}

\begin{lemma}~\label{key}
Let $u$ and $v$ be the real part and the imaginary part of a smooth function $f\co \C^m \to \C$, respectively. 
Then, $\partial u \wedge \partial v=\frac{i}{2}\partial f\wedge (\overline {\bar{\partial }f})$. 
\end{lemma}
\begin{proof}
Since $f=u+iv$, we have $\partial f=\partial u+i\partial v$ and 
$\overline{\bar{\partial }f}=\partial {\bar{f}}=\partial {u}-i\partial {v}$. 
Hence, 
\begin{eqnarray*}
\partial f\wedge (\overline {\bar{\partial }f})
=(\partial u+i\partial v)\wedge (\partial {u}-i\partial {v})=-2i\partial u \wedge \partial v.
\end{eqnarray*}
\end{proof}

In \cite{MR2304588}, Jacobowitz and Landweber have given a necessary and sufficient condition 
for an embedding to be a CR regular embedding. \\

\begin{proposition}[Jacobowitz-Landweber \cite{MR2304588}]
An embedding $F\co M^{2n+k}\to \C^{n+k}$ is a CR regular embedding if and only if 
the submanifold $F(M^{2n+k})$ is given by simultaneous real equations 
$$\rho_j (z_1, z_2, \ldots , z_{n+k})=0 \; \; (j=1, \ldots , k) $$ 
satisfying $\partial \rho_1\wedge \cdots \wedge \partial \rho_k\ne 0$ at each point of $F(M^{2n+k})$. 
\end{proposition}

Applying this proposition to the case where the submanifold is the graph of a function, we obtain the following. \\

\begin{proposition}~\label{ns condition}
Let $f_j$ $(j=1, \ldots , q)$ be smooth complex-valued functions on $\C^m$ with $1\leq q \leq m-1$. 
The embedding $F\co S^{2m-1}\to \C^{m+q}$ defined by 
$$F(z)=\Big(z, f_1(z), \ldots , f_q(z)\Big)$$ 
is a CR regular embedding if and only if the $(q+1)$ complex vectors 
$z$, $\dfrac{\partial f_j}{\partial \bar{z}}(z)$ $(j=1, \ldots , q)$ 
are linearly independent over $\C$ for each $z\in S^{2m-1}$. 
\end{proposition}
\begin{proof}
The submanifold $F(S^{2m-1})$ is described as 
$$\left\{(z, z_{m+1}, \ldots , z_{m+q})\in \C^{m+q} \mid ||z||^2=1, z_{m+1}=f_1(z), \ldots , z_{m+q}=f_q(z) \right\}.$$
We define smooth real functions $\rho_1(z)$, $\cdots $, $\rho_{2q+1}(z)$ by 
\begin{eqnarray*}
\rho _{1}=-1+\sum_{k=1}^mz_k\bar{z}_k, \; z_{m+j}-f_j(z) =\rho_{2j}+i \rho_{2j+1}\;  (j=1, \ldots , q),  
\end{eqnarray*}
for which we have 
$$F(S^{2m-1})=\rho_1^{-1}(0)\cap \rho_2^{-1}(0) \cap \cdots \cap \rho_{2q+1}^{-1}(0).$$
By Lemma~\ref{key}, 
\begin{eqnarray*}
\partial \rho_{2j}\wedge \partial \rho_{2j+1}
=\dfrac{i}{2}\partial (z_{m+j}-f_j(z))\wedge \overline{\bar{\partial } (z_{m+j}-f_j(z))}
=\dfrac{i}{2}\overline{(\bar{\partial } f_j)}\wedge (dz_{m+j}-\partial f_j). 
\end{eqnarray*}
Therefore, $\partial \rho_1 \wedge \partial \rho_2 \wedge \cdots \wedge \partial \rho_{2q+1} \ne 0$ holds if and only if 
\begin{eqnarray*}
(\bar{z}_1dz_1+\cdots +\bar{z}_mdz_m)
\wedge 
\Big((\overline{\frac{\partial f_1}{\partial \bar{z}_1}}) dz_1+\cdots +(\overline{\frac{\partial f_1}{\partial \bar{z}_m}}) dz_m\Big)
\wedge \cdots 
\wedge 
\Big((\overline{\frac{\partial f_q}{\partial \bar{z}_1}}) dz_1+\cdots +(\overline{\frac{\partial f_q}{\partial  \bar{z}_m}})dz_m\Big)
\ne 0
\end{eqnarray*}
holds. 
This condition is equivalent to the complex vectors 
$$z=(z_1, \ldots , z_m), 
\frac{\partial f_1}{\partial \bar{z}}=
\Big({\frac{\partial f_1}{\partial \bar{z}_1}}, \ldots ,{\frac{\partial f_1}{\partial \bar{z}_m}}\Big), 
\ldots , 
\frac{\partial f_q}{\partial \bar{z}}=
\Big({\frac{\partial f_q}{\partial \bar{z}_1}}, \ldots ,{\frac{\partial f_q}{\partial \bar{z}_m}}\Big)$$ 
being linearly independent over $\C$. 
\end{proof}

Now, we are ready to prove our main theorem. 
First we reprove Ahern-Rudin's result 
from the viewpoint of Proposition~\ref{ns condition}, 
and then, prove Theorem~\ref{main}. 

\begin{proof}[Proof of Theorem~\ref{AR}]
When $(z_1, z_2)\ne (0, 0)$, 
the two vectors $(z_1, z_2)$ and $(\dfrac{\partial P}{\partial \bar{z}_1}, \dfrac{\partial P}{\partial \bar{z}_2})$ 
are linearly independent over $\C$. Indeed, the function 
$$z_2\dfrac{\partial P}{\partial \bar{z}_1} - z_1\dfrac{\partial P}{\partial \bar{z}_2}
=|z_2|^2(|z_2|^2-2|z_1|^2)-i|z_1|^2(|z_1|^2-2|z_2|^2)$$ 
vanishes only at the origin $(0,0)\in \C^2$. 
Therefore, by Proposition~\ref{ns condition}, the embedding $F$ is a totally real embedding. 
\end{proof}

\begin{proof}[Proof of Theorem~\ref{main}] 
Suppose $z=(z_1, z_2, \ldots , z_{2n-1}, z_{2n})\ne (0,0, \ldots , 0,0)$.  
Then there exists $j$ such that $(z_{2j-1}, z_{2j})\ne (0, 0)$. 
For such a $j$, the two vectors $(z_{2j-1}, z_{2j})$ and $\dfrac{\partial P}{\partial \bar{z}}(z_{2j-1}, z_{2j})$ 
are linearly independent over $\C$ by the proof of Theorem~\ref{AR}. 
Hence, the two vectors $z=(z_1, z_2, \ldots , z_{2n-1}, z_{2n})$ and 
$$\frac{\partial Q}{\partial \bar{z}}(z)
=(\frac{\partial P}{\partial \bar{z}}(z_1, z_2), \ldots, \frac{\partial P}{\partial \bar{z}}(z_{2n-1}, z_{2n}))$$  
are linearly independent over $\C$. 
Then, by Proposition~\ref{ns condition}, the embedding $F$ is CR regular. 
\end{proof}

\section*{Acknowledgements}
The author has been  
supported in part by the Grant-in-Aid for Young Scientists (B), 
No.~17K14193, Japan Society for the Promotion of Science. 
%

\begin{thebibliography}{10}

\bibitem{MR787894}
Patrick Ahern and Walter Rudin, 
\emph{Totally real embeddings of {$S^3$} in {${\bf C}^3$}}, 
Proc. Amer. Math. Soc. \textbf{94} (1985), no.~3, 460--462. \MR{787894}

\bibitem{MR966952}
Mich\`ele Audin, 
\emph{Fibr\'es normaux d'immersions en dimension double, 
points doubles d'immersions lagragiennes et plongements totalement r\'eels}, 
Comment. Math. Helv. \textbf{63} (1988), no.~4, 593--623. 

\bibitem{MR1553196}
Elie Cartan, 
\emph{Sur la g\'eom\'etrie pseudo-conforme des hypersurfaces de l'espace de deux variables complexes}, 
Ann. Mat. Pura Appl. \textbf{11} (1933), no.~1, 17--90. 

\bibitem{1607.02384}
Ali~M. Elgindi, 
\emph{On the non-existence of CR-regular embeddings of {$S^5$}}, 
Complex Var. Elliptic Equ. \textbf{64} (2019), no.~9, 1564--1567. 

\bibitem{MR880125}
Franc Forstneri{\v{c}}, 
\emph{On totally real embeddings into {${\bf C}^n$}}, 
Exposition. Math. \textbf{4} (1986), no.~3, 243--255. 

\bibitem{MR0420697}
M.~L. Gromov, \emph{A topological technique for the construction of solutions
  of differential equations and inequalities}, Proceedings ICM (Nice 1970), vol. 2 (1971), 221--225. 

\bibitem{MR0413206}
M.~L. Gromov, \emph{Convex integration of differential relations. {I}}, Izv. Akad.
  Nauk SSSR Ser. Mat. \textbf{37} (1973), 329--343. 

\bibitem{MR864505}
Mikhael Gromov, \emph{Partial differential relations}, Ergebnisse der
  Mathematik und ihrer Grenzgebiete (3) [Results in Mathematics and Related
  Areas (3)], vol.~9, Springer-Verlag, Berlin, 1986. 

\bibitem{MR2304588}
Howard Jacobowitz and Peter Landweber, 
\emph{Manifolds admitting generic immersions into {$\Bbb C^N$}}, 
Asian J. Math. \textbf{11} (2007), no.~1, 151--165. 

\bibitem{MR3871721}
Naohiko Kasuya and Masamichi Takase, 
\emph{Generic immersions and totally real embeddings}, 
Internat. J. Math. \textbf{29} (2018), no.~11, 1850073. 

\bibitem{Erratum}
Naohiko Kasuya and Masamichi Takase, 
\emph{Erratum to ``Generic immersions and totally real embeddings'' }, 
Int. J. Math. (2019), online. 
DOI: 10.1142/S0129167X19920034

\bibitem{MR0314066}
Hon~Fei Lai, 
\emph{Characteristic classes of real manifolds immersed in complex manifolds}, 
Trans. Amer. Math. Soc. \textbf{172} (1972), 1--33. 

\bibitem{MR0410764}
J.~Alexander Lees, \emph{On the classification of {L}agrange immersions}, 
Duke Math. J. \textbf{43} (1976), no.~2, 217--224. 

\bibitem{MR3096869}
Marko Slapar, \emph{Cancelling complex points in codimension two}, 
Bull. Aust. Math. Soc. \textbf{88} (2013), no.~1, 64--69. 

\bibitem{MR3345510}
Marko Slapar, 
\emph{C{R} regular embeddings and immersions of compact orientable 4-manifolds into {$\Bbb{C}^3$}}, 
Internat. J. Math. \textbf{26} (2015), no.~5, 1550033.  

\bibitem{MR0145555}
Noboru Tanaka, 
\emph{On the pseudo-conformal geometry of hypersurfaces of the space of {$n$}\ complex variables}, 
J. Math. Soc. Japan \textbf{14} (1962), 397--429. 

\bibitem{MR0221418}
Noboru Tanaka, 
\emph{On generalized graded {L}ie algebras and geometric structures. {I}}, 
J. Math. Soc. Japan \textbf{19} (1967), 215--254. 
  
\bibitem{MR3503716}
Rafael Torres, \emph{C{R} regular embeddings and immersions of 6-manifolds into
  complex 4-space}, Proc. Amer. Math. Soc. \textbf{144} (2016), no.~8,
  3493--3498. 

\bibitem{MR3777134}
Rafael Torres, \emph{An equivalence between pseudo-holomorphic embeddings into
  almost-complex {E}uclidean space and {CR} regular embeddings into complex
  space}, Enseign. Math. \textbf{63} (2017), no.~1, 165--180. 

\bibitem{MR0241690}
R.~O. Wells, Jr., \emph{Holomorphic hulls and holomorphic convexity}, Rice
  Univ. Studies \textbf{54} (1968), no.~4, 75--84. 

\bibitem{MR0237823}
R.~O. Wells, Jr., \emph{Compact real submanifolds of a complex manifold with
  nondegenerate holomorphic tangent bundles}, Math. Ann. \textbf{179} (1969),
  123--129. 
\end{thebibliography}


\end{document}